\documentclass{amsart}
\usepackage{amsfonts}
\usepackage{amsmath,amsthm}
\usepackage{amsfonts,amssymb}
\usepackage{mathrsfs}
\usepackage{enumerate}

\newtheorem{thm}{Theorem}[section]
\newtheorem{cor}[thm]{Corollary}
\newtheorem{lem}[thm]{Lemma}

\theoremstyle{remark}

\numberwithin{equation}{section}

\def \b{{\bf b}}

\def\vz{\varepsilon}
\def\oz{\omega}
\def\lz{\lambda}

\def\dz{\delta}
\def\az{\alpha}

\def\({\Bigl(}
\def \){ \Bigr)}

 \def\a{{\bf a}}

 \def\h{{\bf h}}

\def\x{{\bf x}}
\def\h{{\bf h}}
\def\y{{\bf y}}

\def\va{\varepsilon}

\def\y{{\bf y}}

\begin{document}

\title[] {A note about EC-$(s,t)$-weak tractability  of  multivariate  approximation with analytic Korobov  kernels}

\author[]{Heping Wang} \address{School of Mathematical Sciences, Capital Normal
University, Beijing 100048,
 China.}
\email{wanghp@cnu.edu.cn.}

%\date{\today}


\thanks{
 Supported by the
National Natural Science Foundation of China (Project no.
11671271) and
 the  Beijing Natural Science Foundation (1172004)
 }

\begin{abstract} This note is devoted to discussing multivariate approximation of
continuous functions on $[0,1]^d$ with analytic Korobov  kernels
in the worst and average case settings.  We only consider
algorithms that use finitely many evaluations of arbitrary
continuous linear functionals.  We study EC-$(s, t)$-weak
tractability   under the absolute   or normalized error criterion,
and obtain  necessary and sufficient conditions for
$0<\min(s,t)<1$ and $\max(s,t)\le 1$ in the worst case setting and
for $s,t>0$ in the average case setting.

\end{abstract}

\maketitle
\input amssym.def

\section{Introduction and main results}

We approximate multivariate problems $S=\{S_d\}_{d\in\Bbb N}$ by
algorithms that use finitely many linear functionals. The
information complexity $n(\vz, S_d)$ is defined as the minimal
number of linear functionals which are needed to find an
approximation to within an error threshold $\vz$.

 We consider exponentially-convergent tractability (EC-tractability) of the multivariate problems $S=\{S_d\}$. There are two kinds of tractability based
on polynomial-convergence and exponential-convergence. The
classical tractability describes how the information complexity
 behaves as a function of $d$ and $\vz^{-1}$, while the
exponentially-convergent tractability (EC-tractability) does as
one of $d$ and $(1+\ln\vz^{-1})$. Nowadays study of tractability
and EC-tractability has become one of the busiest areas of
research in information-based complexity (see \cite{NW1, NW2, NW3,
DLPW, IKPW, PP, X2} and the references therein).

 We briefly recall the basic
EC-tractability notions. Let $S= \{S_d\}_{d\in\Bbb N}$. We say $S$
is

$\bullet$   {\it Exponential convergent and strong polynomial
tractable  (EC-SPT)}   iff there exist non-negative numbers $C$
and $p$ such that for all $d\in \Bbb N,\ \va \in (0,1)$,
\begin{equation*}
n(\va ,S_d)\leq C(1+\ln \va ^{-1})^p;
\end{equation*}

 $\bullet$  {\it Exponential convergent and  polynomial
tractable  (EC-PT)} iff there exist non-negative numbers $C, p$
and $q$ such that for all $d\in \Bbb N, \ \va \in(0,1)$,
\begin{equation*}
n(\va ,S_d)\leq Cd^q(1+\ln \va ^{-1})^p;
\end{equation*}

$\bullet$   {\it Exponential convergent and quasi-polynomial
tractable (EC-QPT)} iff there exist two constants $C,t>0$ such
that for all $d\in \Bbb N, \ \va \in(0,1)$,
\begin{equation*}
n(\va ,S_d)\leq C\exp\{t[1+\ln(1+\ln \va ^{-1})](1+\ln d)\};
\end{equation*}

$\bullet$ {\it Exponential convergent and uniformly weakly
tractable (EC-UWT)} iff for all $\az, \beta>0$,
\begin{equation*}
\lim_{\varepsilon ^{-1}+d\rightarrow \infty }\frac{\ln n(\va
,S_d)}{(\ln \va ^{-1})^{\az }+d^{\beta }}=0;
\end{equation*}

$\bullet$  {\it Exponential convergent and weakly tractable
(EC-WT)} iff
\begin{equation*}
\lim_{\va ^{-1}+d\rightarrow \infty }\frac{\ln n(\va ,S_d)}{\ln\va
^{-1}+d}=0;
\end{equation*}

$\bullet$ {\it Exponential convergent and $(s,t)$-weakly tractable
(EC-$(s,t)$-WT)} for positive $s$ and $t$ iff
\begin{equation*}
\lim_{\varepsilon ^{-1}+d\rightarrow \infty }\frac{\ln n(\va
,S_d)}{(\ln \va ^{-1})^{s }+d^{t }}=0.
\end{equation*}

Clearly, EC-$(1,1)$-WT is the same as EC-WT, and for $0<s_1<s, \
0<t_1<t$,  EC-$(s_1,t_1)$-WT $\Longrightarrow$ EC-$(s,t)$-WT. We
also have

\vskip 2mm
\begin{center}EC-SPT $\Longrightarrow$ EC-PT  $\Longrightarrow$  EC-QPT
$\Longrightarrow$ EC-UWT  $\Longrightarrow$     EC-WT.\end{center}

\vskip 2mm

 In the definitions of EC-SPT, EC-PT, EC-QPT, EC-UWT, EC-WT, and
EC-$(s,t)$-WT, if we replace $(1+\ln {\vz}^{-1})$ by $\vz^{-1}$,
we get the definitions of \emph{strong polynomial tractability
(SPT)}, \emph{polynomial tractability (PT)},
\emph{quasi-polynomial tractability (QPT)}, \emph{uniform weak
tractability
 (UWT)}, \emph{weak tractability
 (WT)}, and \emph{$(s,t)$-weak tractability
 ($(s,t)$-WT)}, respectively.

 \vskip 2mm

This note is devoted to discussing EC-$(s,t)$-WT of  multivariate
approximation with analytic Korobov kernels in the worst and
average case settings.

Let ${\bf a}=\big\{a_k\big\}_{k\ge1}$  be a non-decreasing
sequence of positive  numbers, and let ${\bf
b}=\big\{b_k\big\}_{k\ge1}$ be a sequence of positive numbers
having a positive infimum $b_*$ so that
\begin{equation}\label{1.1}0<a_1\le a_2 \le \dots\le a_k\le \dots ,\ \ \ {\rm and}\  \ \ b_*:=\inf_{k\ge 1}
b_k>0.\end{equation}Assume that the analytic Korobov kernel
$K_{d,\a,\b}$ is of product form,
\begin{equation}\label{1.1-1}K_{d,\a,\b}(\x,\y)=\prod_{k=1}^d K _{1,a_k,b_k}(x_k,y_k),\quad \x,\,\y\in
[0,1]^d,\end{equation} where $ K _{1,a,b}$ are univariate analytic
Korobov kernels,
$$K_{1,a,b}(x,y)=\sum_{\h\in\Bbb Z}\oz^{a|h|^b}\exp(2\pi i h(x-y)), \
\ x, y\in[0,1].$$Here $\oz\in(0,1)$ is a fixed positive number,
$i=\sqrt{-1}$,  $a,\, b>0$. Hence, we have
\begin{equation}\label{1.2}K_{d,\a,\b}(\x, \y)=\sum_{\h\in\Bbb Z^d}\oz_\h\exp(2\pi i\h\cdot(\x-\y)),\ \ \x, \y\in[0,1]^d,\end{equation}
where
\begin{equation}\label{1.3}\oz_\h=\oz^{\sum_{k=1}^da_k|h_k|^{b_k}},\end{equation} for
fixed $\oz\in(0,1)$ and all $ \h=(h_1,h_2,\dots,h_d)\in\Bbb Z^d$,
and
$$\x\cdot \y=\sum_{k=1}^dx_ky_k,\ \ \ \x=(x_1,x_2,\cdots,x_d),\
\y=(y_1,y_2,\cdots,y_d)\in \Bbb R^d
$$ denotes the usual Euclidean inner product.

First we consider  the worst case setting. Denote by
$H(K_{d,\a,\b})$ the analytic Korobov space which is a reproducing
kernel Hilbert space with the reproducing kernel $K_{d,\a,\b}$
given by \eqref{1.2}. Such  space $H(K_{d,\a,\b})$ has been widely
used in  tractability study (see \cite{DKPW, DLPW, IKPW, KPW0,
KPW, LX1}).

We consider multivariate approximation problem ${\rm APP}=\{{\rm
APP}_d\}_{d\in \Bbb N}$ which is defined via the embedding
operator
\begin{equation}\label{1.3-1} {{\rm APP}}_d: H(K_{d,\a,\b})\to L_{2}([0,1]^d)\ \ {\rm with}\ \  {\rm APP}_d\,
f=f.\end{equation} We approximate ${\rm APP}_d$ by algorithms
that use only finitely many continuous linear functionals on
$H(K_{d,\a,\b})$. A function $f\in H(K_{d,\a,\b})$ is approximated
by an algorithm
\begin{equation}\label{1.4}A_{n,d}f=\phi
_{n,d}(L_1(f),L_2(f),\dots,L_n(f)),\end{equation} where
$L_1,L_2,\dots,L_n$ are continuous linear functionals on
$H(K_{d,\a,\b})$,
 and $\phi _{n,d}:\;\Bbb R^n\to
L_2([0,1]^d)$ is an arbitrary measurable mapping. The worst case
error of approximation by an algorithm $A_{n,d}$ of the form
\eqref{1.4}  is defined as
 $$e^{\rm wor}(A_{n,d})=\sup_{\|f\|_{H(K_{d,\a,\b})}\le1}
\|{\rm APP}_d\,f-A_{n,d}f\|_{L_{2}([0,1]^d)}.$$ The $n$th minimal
worst case error, for $n\ge 1$, is defined by
$$e^{\rm wor}(n,d)=\inf_{A_{n,d}}e^{\rm wor}(A_{n,d}),$$
where the infimum is taken over all algorithms of the form
\eqref{1.4}.

For $n=0$, we use $A_{0,d}=0$. We remark that  the so-called
initial error $e^{\rm wor}(0,d )$, defined by
$$e^{\rm wor}(0,d )=\sup_{\|f\|_{H(K_{d,\a,\b})}\le1}\|{\rm APP}_d\,
f\|_{L_{2}([0,1]^d)},$$ is equal to $1$. In other words, the
normalized error criterion and the absolute error criterion
coincide.

The information complexity $n(\vz, d)$ is defined by
\begin{equation*}
n(\vz ,d)=\min\{n:\,e^{\rm wor}(n,d)\leq \vz \}.
\end{equation*}

The classical tractability of the multivariate problem APP has
been investigated  and solved completely in \cite{KPW, LX1, IKPW}.
For the EC-tractability of  APP,   the sufficient and necessary
conditions for  EC-SPT, EC-PT, EC-QPT, EC-UWT, EC-WT, and
EC-$(s,t)$-WT with $\max(s,t)>1$
   were given in
  \cite{IKPW}. See the following  EC-tractability  results  of  APP:

  \vskip 2mm

 $\bullet$  EC-SPT holds iff EC-PT holds iff
$$\sum_{k=1}^\infty b_k^{-1}<\infty  \qquad {\rm and}\qquad\underset{k\to\infty}{\underline{\lim}}\frac{\ln a_k}{
k}>0.$$

\vskip 1mm

 $\bullet$  EC-QTP holds iff $$\sup_{d\in \Bbb N}\frac{\sum_{k=1}^d b_k^{-1}}{1+\ln d}<\infty\qquad {\rm and}\qquad
 \underset{k\to\infty}{\underline{\lim}}\frac{(1+\ln k)\ln a_k}{
k}>0.$$

\vskip 1mm

 $\bullet$  EC-UWT holds iff $$\lim_{k\to\infty}\frac{\ln a_k}{\ln k}=\infty.$$

\vskip 1mm

$\bullet$ EC-$(s,t)$-WT with $\max(s,t)>1$ always holds.

\vskip 3mm

 $\bullet$ EC-WT  holds iff WT holds iff  $$\lim\limits_{k\to \infty}a_k=\infty.$$

\vskip 1mm

However, the authors did not find out the
 conditions on EC-$(s,t)$-WT with
$\max(s,t)\le 1$ and $\min(s,t)<1$ in \cite{IKPW}. In this note,
we fill the gap and obtain the sufficient and necessary conditions
for EC-$(s,t)$-WT with $\max(s,t)\le 1$ and $\min(s,t)<1$. We use
estimates of  entropy numbers and technique  in \cite{SW} to
obtain the sufficient conditions for EC-$(s,t)$-WT. Such method is
first used in \cite{KMU}.

\begin{thm}\label{thm1} Consider the   approximation problem
APP  in the worst case setting with the sequences $\a$ and $\b$
 satisfying \eqref{1.1}. Then
\vskip 1mm

 (i) EC-$(1,t)$-WT with $t<1$ holds iff
\begin{equation}\label{1.7}\lim\limits_{j\to\infty}\frac{\ln j}{ a_j}=0.\end{equation}

\vskip 1mm

(ii)  EC-$(s,t)$-WT with $s<1$ and $t\leq 1$  holds iff
\begin{equation}\label{1.8}\lim\limits_{j\to\infty}\frac{j^{(1-s)/s}}{ a_j}=0.\end{equation}

\end{thm}

Next we discuss the average case setting. Consider the
approximation problem $I=\{I_d\}_{d\in\Bbb N}$,
\begin{equation}\label{1.9} I_d\,\,: \,\,C([0,1]^d)\to L_2([0,1]^d)\quad
\text{with} \quad I_df=f.\end{equation}

The space $C([0,1]^d)$ of continuous real functions is equipped
with a zero-mean Gaussian measure $\mu_d$ whose covariance kernel
is given by  the analytic Korobov kernel $K_{d,\a,\b}$. We
approximate $I_d \, f$ by algorithms $A_{n,d}f$ of the form
\eqref{1.4} that use $n$ continuous linear functionals on
$C([0,1]^d)$.
  The average case error
for $A_{n,d}$ is defined by
\begin{equation*}
e^{\rm avg}(A_{n,d})\;=\;\Big[ \int _{C([0,1]^d)}\big \|
I_d\,f-A_{n,d}f \big \|_{L_2([0,1]^d)}^{2}\mu _d(df) \Big
]^{\frac{1}{2}}.
\end{equation*}

The $n$th minimal average case error, for $n\ge 1$, is defined  by
\begin{equation*}
e^{\rm avg}(n,d)=\inf_{A_{n,d}}e(A_{n,d} ),
\end{equation*}
where the infimum is taken over all algorithms of the form
\eqref{1.4}.

For $n=0$, we use $A_{0,d}=0$. We obtain   the so-called initial
error $$e^{\rm avg}(0,d)=e^{\rm avg}(A_{0,d}) .$$

 The information
complexity for $I_d$ can be studied using either the absolute
error criterion (ABS), or the normalized error criterion (NOR).
Then we define the information complexity $n^{{\rm avg},X}(\va
,d)$ for $X\in \{ {\rm ABS,\, NOR}\}$ as
\begin{equation*}
n^{{\rm avg}, X}(\va ,d)=\min\{n:\,e^{\rm avg}(n,d)\leq \va
CRI_d\},
\end{equation*}where
\begin{equation*}
CRI_d=\left\{\begin{matrix}
 & 1, \; \ \ \quad\qquad\text{ for X=ABS,} \\
 &e^{\rm avg}(0,d),\quad \text{ for X=NOR.}
\end{matrix}\right.
\end{equation*}

The classical tractability of the multivariate problem $I=\{I_d\}$
has been investigated   in \cite{LX2, LX3, CWZ}. For the
EC-tractability of  $I$,   the sufficient and necessary conditions
for  EC-SPT, EC-PT, EC-UWT, EC-WT under ABS or NOR
   were given in
  \cite{LX2}, see the following  EC-tractability  results  of  $I$:

  \vskip 2mm

 $\bullet$ For ABS or NOR,  EC-SPT holds iff EC-PT holds iff
$$\sum_{k=1}^\infty b_k^{-1}<\infty  \qquad {\rm and}\qquad\underset{k\to\infty}{\underline{\lim}}\frac{\ln a_k}{
k}>0.$$

\vskip 1mm

 $\bullet$ For ABS or NOR,   EC-UWT holds iff $$\lim_{k\to\infty}\frac{\ln a_k}{\ln k}=\infty.$$

\vskip 1mm

 $\bullet$  For ABS or NOR,  EC-WT  holds  iff  $$\lim\limits_{k\to \infty}a_k=\infty.$$

\vskip 1mm

 In this note,
we  obtain the sufficient and necessary conditions for
EC-$(s,t)$-WT. We use the connection about EC-tractability in the
worst and average case settings (see \cite{X2, LXD}). Such
connection was used to study the  EC-tractability of multivariate
approximation with Gaussian kernel in the average case setting
(see \cite{CW}). Specially, according to \cite[Theorems 3.2 and
4.2]{X2} and \cite[Theorem 3.2]{LXD}, we have the same results in
the worst and average case settings  concerning EC-WT, EC-UWT, and
EC-$(s,t)$-WT for $0<s\le 1$ and $t>0$ under ABS.

\begin{thm}\label{thm2} Consider the above  approximation problem
$I=\{I_d\}$  with the sequences $\a$ and $\b$
 satisfying \eqref{1.1}. Then

\vskip 2mm

(i)  for ABS or NOR, if $s>0$ and $t>1$ then EC-$(s,t)$-WT always
holds;

\vskip 2mm

(ii)  for ABS or NOR, EC-$(s,1)$-WT with $s\ge 1$ holds iff EC-WT
holds iff $$\lim\limits_{j\to \infty} a_j=\infty;$$

(iii)  for ABS, EC-$(1,t)$-WT with $t<1$ holds iff
\begin{equation}\label{1.10}\lim\limits_{j\to\infty}\frac{\ln j}{ a_j}=0;\end{equation}

(iv)  for ABS or NOR, EC-$(s,t)$-WT with $s<1$ and $t\leq 1$
holds iff
\begin{equation}\label{1.11}\lim\limits_{j\to\infty}\frac{j^{(1-s)/s}}{a_j}=0;\end{equation}

 (v)  for ABS or NOR, EC-$(s,t)$-WT with $s>1$ and $t< 1$ holds iff
\begin{equation}\label{1.12}\lim_{j\to
\infty}j^{1-t}a_j\,\oz^{a_j}=0.\end{equation}

\end{thm}

The paper is organized as follows. In Section 2 we give some
necessary preliminaries  in the worst and average case settings.
     In  Section 3, we give the proofs of Theorems 1.1 and 1.2.

\section{Preliminaries}

For a fixed $\oz\in (0,1)$, let $K_{d,\a,\b}$ be the analytic
Korobov kernel given by \eqref{1.2} with $\a,\b$ satisfying
\eqref{1.1}. By \eqref{1.1-1} we know that the reproducing kernel
Hilbert space $H(K_{d,\a,\b} )$ is a tensor product of the
univariate reproducing kernel Hilbert spaces $H(K_{1,a_j,b_j}),\
j=1,\dots,d$ with reproducing kernels $K_{1,a_j,b_j}$, i.e.,
$$H(K_{d,\a,\b} )=H(K_{1,a_1,b_1})\otimes H(K_{1,a_2,b_2})\otimes
H(K_{1,a_d,b_d}).$$

From \cite{NW1} we know  that $e^{\rm wor}(n,d)$ depends on the
eigenpairs of the operator
$$W_d={\rm APP}_d^*\,{\rm APP}_d : H(K_{d,\a,\b} )\mapsto H(K_{d,\a,\b}
),$$where ${\rm APP}_d$ is given by \eqref{1.3-1}. We have
$$W_df=\sum_{\h\in\Bbb Z^d}\oz_\h \langle f, e_\h\rangle_{H(K_{d,\a,\b}
)}\,e_\h$$with $$e_\h(\x)=(\oz_\h)^{1/2}\exp(2\pi i \h\cdot \x).$$
This means that $\{(\oz_\h, e_\h)\}_{\h\in\Bbb Z^d}$ are the
eigenpairs of $W_d$, i.e.,
$$W_d\, e_\h=\oz_\h\, e_\h, \ \ \ {\rm for \ all}\ \h\in\Bbb
Z^d,$$and $\{e_\h\}_{\h\in \Bbb Z^d}$ is an orthonormal basis for
$H(K_{d,\a,\b} )$.

Let $\{(\lz_{d,j},\eta_{d,j})\}_{j\in\Bbb N}$ be the rearrangement
of the eigenpairs $\{(\oz_\h, e_\h)\}_{\h\in\Bbb Z^d}$, such that
the eigenvalues $\oz_\h,\ \h\in \Bbb Z^d$ are arranged in
decreasing order, i.e.,
$$\lz_{d,1}\ge \lz_{d,2}\ge \cdots\lz_{d,k}\ge\cdots\ge 0.$$

From \cite[p. 118]{NW1}
 we get   that the $n$th minimal worst case
error is
$$e^{\rm wor}(n,d)=(\lz_{d,n+1})^{1/2},$$
and  it is achieved by the algorithm
$$A_{n,d}^*f=\sum_{k=1}^n \lz_{d,k}\langle f, \eta_{d,k} \rangle_{H(K_{d,\a,\b}
)}\, \eta_{d,k}.$$

Since $\lz_{d,1}=\oz_0=\prod_{k=1}^d\lz(k,1)=1$, we get that  the
normalized error criterion and the absolute error criterion
coincide. Then the information complexity $n(\vz,d)$ of APP
satisfies
$$n(\vz,d)=\min\{n\in\Bbb N\ |\ e^{\rm wor}(n,d)\le \vz\}=\min\{n\in\Bbb N\ |\ \lz_{d,n+1}\le \vz^2\}, $$
or equivalently, the number of eigenvalues
$\{\lz_{d,j}\}_{j\in\Bbb N}=\{\oz_\h\}_{\h\in\Bbb Z^d}$ of the
operator $W_d$ greater than $\vz^2$. Due to \eqref{1.3}, we can
rewrite the information complexity as
\begin{align}\label{2.1}n(\vz,d)&=\#\big\{\h\in \Bbb Z^d\ |\ \oz_{\h}=\oz^{\sum_{k=1}^d a_k|h_k|^{b_k}} >\vz^2\big\}\notag\\ &=\#\Big\{\h\in \Bbb Z^d\ |\
\sum\limits_{k=1}^d a_k|h_k|^{b_k}<\frac{\ln \vz^{-2}}{\ln
\oz^{-1}}\Big\}, \end{align} where $\# A$ represents the number of
elements in a set $A$.

\vskip 1mm

Next we can give explicit formulas for the $n$th minimal average
case error $e^{avg}(n,d)$ and the corresponding $n$th optimal
algorithm, see \cite[Section 4.3]{NW1}. We recall that the space
$C([0,1]^d)$  is equipped with a zero-mean Gaussian measure
$\mu_d$ whose covariance kernel is given by  the analytic Korobov
kernel $K_{d,\a,\b}$. Let $$C_{\mu_d}:
\big(C([0,1]^d)\big)^*\mapsto C([0,1]^d)$$  denote the covariance
operator of $\mu_d$, as defined in \cite[Appendix B]{NW1}. Then
the induced measure $\nu_d = \mu_d (I_d)^{-1}$ is a zero-mean
Gaussian measure on the Borel sets of $L_2([0,1]^d)$, with
covariance operator $C_{\nu_d}: L_2([0,1]^d)\mapsto C([0,1]^d)$
given by $$ C_{\nu_d}=I_d \,C_{\mu_d}\,(I_d)^*,$$ where $I_d$ is
defined by \eqref{1.9}, $(I_d)^*: L_2([0,1]^d) \mapsto
\big(C([0,1]^d)\big)^*$ is the operator dual to $I_d$. It is
well-known that $C_{\nu_d}$ is a self-adjoint nonnegative-definite
operator with finite trace on $L_2([0,1]^d)$ and for any $f\in
L_2([0,1]^d)$,
$$C_{\nu_d}f(x)=\int_{[0,1]^d} K_{d,\a,\b}(x,y)f(y)dy.$$
Then $\{(\oz_\h, \tilde e_\h)\}_{\h\in\Bbb Z^d}$ are the
eigenpairs of $C_{\nu_d}$ with $\tilde e_\h(\x)=\exp(2\pi i
\h\cdot \x)$, i.e.,
$$C_{\nu_d}\, \tilde e_\h=\oz_\h\, \tilde e_\h, \ \ \ {\rm for \ all}\ \h\in\Bbb
Z^d,$$and $\{\tilde e_\h\}_{\h\in \Bbb Z^d}$ is an orthonormal
basis for $L_2([0,1]^d)$.

Let $\{\lz_{d,j}\}_{j\in \Bbb N}$ be the non-increasing
rearrangement of  $\{\oz_\h\}_{\h\in\Bbb Z^d}$ just as in the
worst case setting. Then the eigenvalues of the covariance
operator $C_{\nu_d}$  are just $\lz_{d,j},\ {j\in \Bbb N}$. Denote
by $\xi_{d,j}$ the corresponding eigenvector of $C_{\nu_d}$ with
respect to the eigenvalue $\lz_{d,j}$.  Then  the $n$th minimal
average case error $e^{\rm avg}(n,d)$ is (see \cite{NW1})
$$e^{\rm avg}(n,d)=\Big(\sum_{k=n+1}^\infty \lz_{d,k}\Big)^{1/2}\ge e^{\rm wor}(n,d).$$
and it is achieved by the algorithm
$$A_{n,d}^{**}f=\sum_{k=1}^n \langle I_d f, \xi_{d,k}
\rangle_{L_2([0,1]^d)}\, \xi_{d,k}.$$

The average case information complexity  can be studied using
either ABS or NOR. Then we define the worst case information
complexity $n^{{\rm wor}, X}(\va ,d)$ for $X\in \{{\rm ABS,\,
NOR}\}$ as
\begin{equation*}
n^{{\rm avg}, X}(\va ,d)=\min\{n:\,e^{\rm avg}(n,d)\leq \va
CRI_d\},
\end{equation*}where
\begin{equation*}
CRI_d=\left\{\begin{split}
 &\ \ 1, \; \quad\qquad\text{ for $X$=ABS,} \\
 &e^{\rm avg}(0,d), \text{ for $X$=NOR}
\end{split}\right. \ \ =\ \ \left\{\begin{split}
 &\ 1, \; \quad\text{ for $X$=ABS,} \\
 &\Big(\sum_{j=1}^\infty \lz_{d,j}\Big)^{1/2},\ \text{ for $X$=NOR.}
\end{split}\right.
\end{equation*}

 Obviously, we have \begin{equation}\label{2.2} n^{\rm avg, NOR}(\va ,d)\le
 n^{\rm avg, ABS}(\vz,d)=n^{\rm avg, NOR}((e^{\rm avg}(0,d))^{-1}\vz ,d).\end{equation}

We remark that the eigenvalues of the  operator $W_d$ or
$C_{\nu_d}$ are given by
$$\big\{\lz_{d,j} \big\}_{j\in \Bbb N}=\big\{\oz_\h\}_{\h\in\Bbb Z^d}=\big\{\lz(1, j_1)\lz(2, j_2)\dots\lz(d, j_d) \big\}_{(j_1,\dots, j_d )\in\Bbb N^d},$$
where $\lz(k,1)=1$, and
$$\lz(k,2j)=\lz(k,2j+1)=\oz^{a_kj^{b_k}},\quad j\in\Bbb N,\ 1\le k\le d.$$
This implies that for any $\tau_0>0$ and $\tau>\tau_0$,
\begin{align*} \sum_{j\in \Bbb
N}\lz^{\tau}_{d,j}&=\prod_{k=1}^d\sum_{j=1}^\infty
\lz(k,j)^\tau=\prod_{k=1}^d\Big(1+2\sum_{j=1}^\infty \oz^{\tau a_k
j^{b_k}}\Big)\\ &=\prod_{k=1}^d\Big(1+\oz^{\tau a_k}
2\sum_{j=1}^\infty \oz^{\tau a_k
(j^{b_k}-1)}\Big)=\prod_{k=1}^d\Big(1+\oz^{\tau
a_k}H(k,\tau)\Big),
\end{align*}where $$1\le H(k,\tau)=2\sum_{j=1}^\infty \oz^{\tau
a_k (j^{b_k}-1)}\le 2\sum_{j=1}^\infty \oz^{\tau a_1
(j^{b_*}-1)}\le 2\sum_{j=1}^\infty \oz^{\tau_0 a_1 (j^{b_*}-1)}.$$
Since
$$\oz^{\tau_0 a_1(j^{b_*}-1)}=j^{-\frac{\tau_0 a_1(j^{b_*}-1)\ln\frac1\oz}{\ln
j}},\ \ {\rm and}\ \ \lim_{j\to\infty}\frac{\tau_0
a_1(j^{b_*}-1)\ln\frac1\oz}{\ln j}=\infty,
$$we get that
$$M_{\tau_0}:=2\sum_{j=1}^\infty \oz^{\tau_0 a_1(j^{b_*}-1)}<\infty.$$
It follows that for any
$\tau>\tau_0>0$,\begin{align}\label{2.3}\ln 2
\sum_{k=1}^d\oz^{\tau a_k}&\le \sum_{k=1}^d \ln(1+\oz^{\tau
a_k})\le
 \ln \Big(\sum_{j\in \Bbb N}\lz^{\tau}_{d,j}\Big)\\ &=\sum_{k=1}^d\ln \big(1+\oz^{\tau a_k}H(k,\tau)) \le \ln(1+M_{\tau_0} \oz^{\tau a_k}) \le M_{\tau_0}
  \sum_{k=1}^d\oz^{\tau a_k},
 \notag\end{align}
where in the first inequality we used the inequality $\ln(1+x)\ge
x\ln2,\ x\in[0,1]$, and in the last inequality we used the
inequality $\ln(1+x)\le x, \ x>0$. By \eqref{2.3} we have
\begin{equation}\label{2.4} \frac{\oz^{a_1}\ln2}2\le\frac{\ln2}2\sum_{k=1}^d\oz^{ a_k}\le  \ln(e^{\rm avg}(0,d))=
\frac12\ln \Big(\sum_{j\in \Bbb N}\lz_{d,j}\Big)\le\frac{ M_1
}2\sum_{k=1}^d\oz^{ a_k}\le \frac{d M_1\oz^{ a_1}
}2.\end{equation}

\section{Proofs of Theorems \ref{thm1} and \ref{thm2}}

In order to prove Theorem 1.1, we shall use  the estimates of
 entropy numbers of   $\ell_p^d$-unit balls with $\ell_\infty^d$-balls. Such method is firstly
used in \cite{KMU}.

Let $\ell_{p}^d\ (0< p\le\infty)$ denote the space $\Bbb R^d$
equipped with the $\ell_{p}^d$-norm  defined by
$$\|\x\|_{\ell_{p}^d}:=\bigg\{\begin{array}{ll}\big(\sum_{i=1}^d
|x_i|^p \big)^{\frac 1p},\ \ \ &0< p<\infty;\\ \max_{1\le i\le d}
|x_i|, &p=\infty.\end{array}$$The unit ball of $\ell_{p}^d$ is
denoted by $B\ell_{p}^d$.

Let $A\subset \Bbb R^d$. An {\it $\vz$-net} for $A$ is a discrete
set of points $\x_1, \x_2,\dots, \x_n$ in $\Bbb R^d$ such that
$$ A\subset \bigcup_{i=1}^n( \x_i+\vz\, B\ell_\infty^d).$$
The covering number $N_\vz(A)$ is the minimal natural number $n$
such that there is an $\vz$-net for $A$ consisting of $n$ points.
Inverse to the covering numbers $N_\vz(A)$ are the (nondyadic)
entropy numbers
$$\vz_n(A,\ell_\infty^d):=\inf\{\vz>0\ |\ N_\vz(A)\le n\}.$$

Points $\y_1,\y_2,\dots,\y_m$ in $\Bbb R^d$ are called {\it
$\vz$-distinguishable} if the $\ell_\infty$ distances between any
two of them exceeds $\vz$, i.e.,
$$\|\y_i-\y_k\|_{\ell_\infty^d}>\vz\ \ \ {\rm for \ all}\  i\neq
k,\ 1\le i,k\le m.$$ Let $M_\vz(A)$ be the maximal natural number
$m$ such that there is an $\vz$-distinguishable set in $A$
consisting of $m$ points. Then we have (see \cite[Chapter 15,
Proposition 1.1]{LGM})
$$M_{2\vz}(A)\le N_\vz(A)\le M_\vz(A).$$ For $A\subset \Bbb R^d$, let $G(A)$ be the grid number of points in $A$ that lie on the grid $\Bbb Z^d$, i.e.,
$$G(A)=\#(A\cap \Bbb Z^d).$$

In the case $A=B\ell_p^d,\ 0<p<\infty$, the behavior in $n$ and
$d$ of the entropy numbers $\vz_n(B\ell_p^d,\ell_\infty^d)$ is
completely understood (see \cite{ET, K, LGM, S}). It follows that
for $0< p<\infty$ and $\vz\in(0,1)$,
\begin{equation}\label{3.0-0}\ln (N_\vz(B\ell_p^d))\le C_p
\bigg\{\begin{array}{ll}\vz^{-p}\ln(2d\vz^p), &d\vz^p\ge 1,\\ d\ln
(2(d\vz^{p})^{-1}), &d\vz^p\le1,
\end{array}\end{equation}where $C(p)$ is depending only on $p$,
but independent of $d$ and $\vz$.

\begin{lem}For $0<p<\infty$ and $m\ge 1$, we have
\begin{equation}\label{3.0} \ln \Big(\#\Big\{\h\in\Bbb Z^d\ \big|\ \sum_{k=1}^d|h_k|^p\le m\Big\}\Big)\le C_p\bigg\{\begin{array}{ll}m\ln (\frac{2d}m), &d\ge m,\\
d\ln (\frac{2m}d), &m\ge d,
\end{array}\end{equation}where $C_p$ is a constant depending only
on $p$, but independent of $d$ and $m$.\end{lem}
\begin{proof} We set
$A=m^{1/p}B\ell_p^d$. Then $$G(A)=\#(A\cap \Bbb
Z^d)=\#\Big\{\h\in\Bbb Z^d\ \big|\ \sum_{k=1}^d|h_k|^p\le
m\Big\}.$$ For $m\ge1$,  $A\cap \Bbb Z^d$ is
$\rho$-indistinguishable for any $\rho\in(1/2,1)$ in $A$. This
means that
$$G(A)\le M_\rho(A)\le N_{\rho/2}(A)\le N_{1/4}(m^{1/p}B\ell_p^d)= N_{m^{-1/p}/4}(B\ell_p^d).$$
By \eqref{3.0-0} we obtain that
$$\ln G(A)\le \ln\Big(N_{m^{-1/p}/4}(B\ell_p^d)\Big)\le C_p\bigg\{\begin{array}{ll}m\ln (\frac{2d}m), &d\ge m,\\
d\ln (\frac{2m}d), &m\ge d.
\end{array} $$Lemma 3.1 is proved.
\end{proof}

\begin{cor}For $0<p<\infty$ and $m\ge 1$, we have
\begin{equation}\label{3.0.0} \ln \Big(\#\Big\{\h\in\Bbb Z^d\ \big|\ \sum_{k=1}^d|h_k|^p\le m\Big\}\Big)\le C_p d\Big(\ln (2d)+\ln (2m)\Big).\end{equation}\end{cor}

\noindent{\it \textbf{Proof of Theorem \ref{thm1}}.}

\

(i) Suppose that EC-$(1,t)$-WT with $t<1$ holds for APP. We want
to show \eqref{1.7}. It follows that  EC-WT  holds also and hence
$\lim\limits_{j\to\infty}a_j=\infty$. By \eqref{2.1} we have
\begin{align*} n(\vz,d)& =\#\Big\{\h\in \Bbb Z^d\ |\
\sum\limits_{k=1}^d a_k|h_k|^{b_k}<\frac{\ln \vz^{-2}}{\ln
\oz^{-1}}\Big\}\notag\\ &\ge \#\Big\{\h\in \{-1, 0,1\}^d\ |\
\sum\limits_{k=1}^d a_k|h_k|<\frac{\ln \vz^{-2}}{\ln
\oz^{-1}}\Big\}\notag \\ &\ge \#\Big\{\h\in \{-1, 0,1\}^d\ |\
\sum\limits_{k=1}^d |h_k|<\frac{\ln \vz^{-2}}{a_d\ln
\oz^{-1}}\Big\}\notag\\  &=\#\Big\{\h\in \{-1, 0,1\}^d\ |\
\sum\limits_{k=1}^d |h_k|\le m \Big\}\notag\\ &=
\bigg\{\begin{matrix}
 \ \ 3^d, &\quad m\ge d ,\\
 \sum_{j=0}^m 2^j\binom dj, \qquad &0\le m\le d,
\end{matrix}  \end{align*}
where $$m=\Big\lceil \frac{\ln \vz^{-2}}{a_d\ln
\oz^{-1}}\Big\rceil-1.$$ It follows by the inequality $$\binom
{m+d}m\ge \max\Big\{ \Big(1+\frac md\Big)^d,\ \Big(1+\frac
dm\Big)^m\Big\}$$that for $1\le m<d$,
\begin{equation}\label{3.1}n(\vz,d)\ge \binom dm \ge \Big(\frac dm\Big)^m.\end{equation}

Set $\vz=\vz_d\in(0,1)$ such that
$$\frac{\ln \vz^{-2}}{a_d\ln
\oz^{-1}}=d^t$$for sufficiently large $d\in\Bbb N$. Then we have
$$m\le \frac{\ln \vz^{-2}}{a_d\ln
\oz^{-1}}=d^t\le m+1. $$This yields $$\ln \frac dm\ge \ln
d^{1-t}=(1-t)\ln d,$$and
$$\ln \vz^{-1}\le \frac12\,\ln\frac 1\oz\, a_d\,(m+1).$$
Since EC-$(1,t)$-WT with $t<1$ holds, we have
\begin{align*}0&=\lim_{\varepsilon ^{-1}+d\rightarrow \infty }\frac{\ln
n(\va ,d)}{\ln \va ^{-1}+d^{t }}\\ &\ge \lim_{d\rightarrow \infty
}\frac{m\ln\frac dm}{\frac12\,\ln\frac 1\oz\, a_d\,(m+1)+ (m+1)}\\
&\ge \lim_{d\to\infty}\frac{(1-t)\ln d}{\frac12\,\ln\frac
1\oz\, a_d\,(1+\frac 1m)+ (1+\frac1m)}\\
&=\lim_{d\to\infty}\frac{(1-t)\ln d}{\frac12\,\ln\frac 1\oz\,
a_d}\ge 0 ,\end{align*} which implies that
$$\lim_{d\to\infty}\frac{\ln d}{
a_d}=0,$$and hence \eqref{1.7}.

Next we suppose that \eqref{1.7} holds. We want to show that
EC-$(1,t)$-WT with $t<1$ holds.  By \eqref{2.1} we have
\begin{align*} n(\vz,d)& =\#\Big\{\h\in \Bbb Z^d\ \big|\
\sum_{k=1}^d a_k|h_k|^{b_k}<\frac{\ln \vz^{-2}}{\ln
\oz^{-1}}\Big\}\le \#\Big\{\h\in \Bbb Z^d\ \big|\
\sum\limits_{k=1}^d a_k|h_k|^{b_*}\le \frac{\ln \vz^{-2}}{\ln
\oz^{-1}}\Big\}\notag
\\ &\le \#\Big\{ \h\in
\Bbb Z^{i-1}\ |\ \sum_{k=1}^{i-1}a_k|h_k|^{b_*}\le \frac{\ln
\vz^{-2}}{\ln \oz^{-1}}\Big\}
 \cdot \#\Big\{\h\in \Bbb Z^{d-i+1}\ |\ \sum_{k=i}^d a_k
|h_k|^{b_*}\le \frac{\ln \vz^{-2}}{\ln \oz^{-1}}\Big\}\notag \\
&\le  \#\Big\{ \h\in \Bbb Z^{i-1}\ |\
\sum_{k=1}^{i-1}|h_k|^{b_*}\le \frac{\ln \vz^{-2}}{a_1\ln
\oz^{-1}}\Big\}\cdot \#\Big\{\h\in \Bbb Z^{d-i+1}\ |\ \sum_{k=i}^d
|h_k|^{b_*}\le \frac{\ln \vz^{-2}}{a_i\ln \oz^{-1}}\Big\}.\notag
\end{align*}
It follows that \begin{align*}\ln\,n(\vz,d)&\le \ln \Big(\#\Big\{
\h\in \Bbb Z^{i-1}\ |\ \sum_{k=1}^{i-1}|h_k|^{b_*}\le \frac{\ln
\vz^{-2}}{a_1\ln \oz^{-1}}\Big\}\Big)\\ & \qquad +\ln \Big(
\#\Big\{\h\in \Bbb Z^{d-i+1}\ |\ \sum_{k=i}^d |h_k|^{b_*}\le
\frac{\ln \vz^{-2}}{a_i\ln \oz^{-1}}\Big\}\Big)\\ &=:
term_1+term_2.\end{align*} By \eqref{3.0.0} we have
\begin{align*}term_1& \le (i-1)\Big\{\ln \big[2(i-1)\big]+ \ln
\Big(\frac{2\ln \vz^{-2}}{a_1\ln \oz^{-1}}\Big)\Big\}.\end{align*}
We set
$$y=\max(d^t,\ln\vz^{-1}),\ \ \ \dz\in (0,1),\ \ \ {\rm and}\ \ \ i=\min(d+1, 1+\lfloor y^{1-\dz}\rfloor).
$$Then we have $$i-1\le y^{1-\dz},\ \ \  \ln \vz^{-1}\le y\le \ln
\vz^{-1}+d^t,$$ and $y\to \infty$ as $\vz^{-1}+d\to\infty$. It
follows that
\begin{align}\label{3.4}\frac{term_1}{\ln \vz^{-1}+d^t}\le \frac
{\ln (2y^{1-\dz})+ \big[\ln (4y)-\ln({a_1\ln
\oz^{-1}})\big]}{y^\dz}\longrightarrow 0,\end{align}as
$y\to\infty$.

Now we deal with $term_2$. Note that if $d\le \lfloor
y^{1-\dz}\rfloor$, then $i=d+1$ and then $term_2=0$. Hence we can
assume that $d>\lfloor y^{1-\dz}\rfloor$. Then $i\le d$ and both
$d$ and $i$ go to infinity with $y$, and hence $a_i\to\infty$.

 If
$m=\frac{\ln \vz^{-2}}{a_i\ln \oz^{-1}}\ge (d-i+1)$, then by
\eqref{3.0} we get
\begin{align} \label{3.5}\frac{term_2}{\ln \vz^{-1}+d^t}&\le\frac {C(d-i+1)\ln (2t)}y\le \frac {C\ln
\vz^{-2} }{ya_i\ln \oz^{-1}}\frac{\ln(2t)}t\le
\frac{2C}{a_i\ln\oz^{-1}}\longrightarrow 0,
\end{align}
as $\vz^{-1}+d\to\infty$, where $t=\frac m{d-i+1}\ge 1$, and in
the last inequality we used $$\ln(2t)\le t\ \ \ {\rm for}\ \
t\ge1.$$

If $m=\frac{\ln \vz^{-2}}{a_i\ln \oz^{-1}}<1$, then $term_2=0$. We
omit this case. If $1\le m=\frac{\ln \vz^{-2}}{a_i\ln \oz^{-1}}\le
(d-i+1)$, then by \eqref{3.0} we get
\begin{align*} \frac{term_2}{\ln \vz^{-1}+d^t}&\le \frac{Cm\ln(\frac{2(d-i+1)}{m})}{y}\le
\frac{2C\ln\big(\frac {2d a_i\ln\oz^{-1}}{\ln\vz^{-2}}
\big)}{a_i\ln\oz^{-1}}\notag\\ &\le \frac{2C}{\ln\oz^{-1}}\cdot
\frac{\ln2+\ln d+\ln a_i +\ln(\ln\oz^{-1})}{a_i},
\end{align*}
Note that $$i=1+\lfloor y^{1-\dz}\rfloor\ge y^{1-\dz}\ge
d^{t(1-\dz)}.$$ It follows by \eqref{1.7} that
$$\lim_{i\to\infty}\frac{\ln d}{a_i}\le \frac1{t(1-\dz)}\lim_{i\to\infty} \frac{\ln i}{a_i}=0.$$
We continue to obtain that
\begin{equation} \label{3.6}\frac{term_2}{\ln \vz^{-1}+d^t}\le \frac{2C}{\ln\oz^{-1}}\cdot
\frac{\ln2+\ln d+\ln a_i +\ln(\ln\oz^{-1})}{a_i}\longrightarrow 0,
\end{equation}as $i\to\infty$. By \eqref{3.4},
\eqref{3.5}, and \eqref{3.6}, we obtain
$$\frac{\ln n(\vz,d)}{\ln \vz^{-1}+d^t}\le \frac{term_1+term_2}{\ln \vz^{-1}+d^t}\longrightarrow
0,$$ as $\vz^{-1}+d\to\infty$. This means that EC-$(1,t)$-WT with
$t<1$ holds for APP if \eqref{1.7} holds. Theorem 1.1 (i) is
proved.

\

(ii) Suppose that EC-$(s,t)$-WT with $s<1$ and $t\leq 1$  holds
for APP. We want to prove \eqref{1.8}. Set $\vz=\vz_d\in(0,1)$ for
sufficiently large $d\in\Bbb N$ such that
$$m\le \frac{\ln \vz^{-2}}{a_d\ln
\oz^{-1}}=\frac d2\le m+1.$$ This gives that $$\frac dm\ge2, \ \
 \ {\rm and} \ \ \ \ln \vz^{-1}\le \frac12\,\ln \oz^{-1}\, a_d\,(m+1).$$

 Since EC-$(s,t)$-WT with $s<1$ and $t\leq 1$  holds,  by \eqref{3.1} we have
\begin{align*}0&=\lim_{\varepsilon^{-1}+d\rightarrow \infty
}\frac{\ln n(\va ,d)}{(\ln \va ^{-1})^s+d^{t }}\\ &\ge
\lim_{d\rightarrow \infty }\frac{m\ln\frac dm}{(\frac12\,\ln\frac
1\oz\, a_d\,(m+1))^s}\\ &\ge \lim_{d\to\infty}\frac{m^{1-s}\ln
2}{(\frac12\,\ln\frac
1\oz\, a_d\,(1+\frac 1m))^s}\\
&=\lim_{d\to\infty}\frac{(\frac d2)^{1-s}\ln 2}{(\frac12\,\ln\frac
1\oz\, a_d)^s}\ge 0 ,\end{align*} which yields that
$$\lim_{d\to\infty}\frac{d^{1-s}}{
a_d^s}=0,$$and hence \eqref{1.8}.

Next we suppose that \eqref{1.8} holds. We want to show that
EC-$(s,t)$-WT with $s<1$ and $t\leq 1$  holds. Set
$$ a_k=k^{\frac{1-s}s}\hat h(k)\ \ \ {\rm and}\ \ \ {\tilde
h(k)}=\inf_{j\ge k}\hat h(j).$$ Then   the sequence $\{{\tilde
h(k)}\}_{k\in\Bbb N}$ is non-decreasing and satisfies $\hat
h(k)\ge \tilde h(k)$ and
$$\lim\limits_{k\to\infty}\tilde
h(k)=\lim\limits_{k\to\infty}\hat h(k)=\infty.$$ We put
$$ h(1)=\tilde h (1),\  \ h(k+1)=\min\{ (1+1/k)h(k), \tilde h(k+1)\}, \ k=1,\dots.$$
Clearly, we have $$h(k)\le (1+1/k)h(k) \ \ {\rm and}\ \  h(k)\le
\tilde h(k)\le \tilde h(k+1),$$ which yields that the sequence
$\{{h(k)}\}_{k\in\Bbb N}$ is non-decreasing. We also note that $$
\hat h(k)\ge \tilde h(k)\ge h(k)$$ and
$$h(2k)\le \frac{2k}{2k-1}h(2k-1)\le \cdots\le \frac{2k}{2k-1} \frac{2k-1}{2k-2}\cdots\frac{k+1}{k}h(k) =2h(k).$$

 If $\vz^{-1}$ is
bounded by  a constant $M$,  then by \eqref{2.1} and \eqref{3.0}
we have
\begin{align*} \frac{\ln n(\vz,d)}{(\ln\vz^{-1})^s+d^t}\le \frac{\ln\big(\#\big\{\h\in \Bbb Z^d\ |\
\sum\limits_{k=1}^d |h_k|^{b_*}\le  \frac{\ln M^2}{a_1\ln
\oz^{-1}}\big\}\big)}{d^t}\le \frac{CM_0\ln \frac{2d}{M_0}}{d^t}
\longrightarrow 0,\end{align*} as $d\to\infty$, where
$M_0=\frac{\ln M^2}{a_1\ln \oz^{-1}}$. In this case, EC-$(s,t)$-WT
with $s<1$ and $t\leq 1$  holds.

Therefore without loss of generality, we may assume that
$\vz^{-1}$ tends to infinity.
 By \eqref{2.1} we get
\begin{align*} n(\vz,d)&\le \#\Big\{\h\in \Bbb Z^d\ |\
\sum\limits_{k=1}^d a_k|h_k|^{b_*}\le \frac{\ln \vz^{-2}}{\ln
\oz^{-1}}\Big\}\notag
\\ &\le \#\Big\{\h\in \Bbb Z^4\ |\
\sum\limits_{k=1}^4 |h_k|^{b_*}\le \frac{\ln \vz^{-2}}{a_1\ln
\oz^{-1}}\Big\}
\\ &\qquad \cdot \#\Big\{\h\in\Bbb Z^{2^{\lceil \log_2 d \rceil}}\ \big|\ \sum_{l=2}^{\lceil \log_2 d \rceil-1}\Big(\sum_{k=2^l+1}^{2^{l+1}}|h_k|^{b_*}\Big)\, a_{2^l}  \le \frac{\ln \vz^{-2}}{\ln
\oz^{-1}}\Big\}\\
&\le  \#\Big\{\h\in \Bbb Z^4\ |\ \sum\limits_{k=1}^4
|h_k|^{b_*}\le \frac{\ln \vz^{-2}}{a_1\ln \oz^{-1}}\Big\}\\
&\qquad \cdot \prod_{l=2}^{\lceil \log_2 d \rceil-1} \#\Big\{\h\in
\Bbb Z^{2^l}\ \big|\ \sum_{k=2^l+1}^{2^{l+1}}|h_k|^{b_*}\le
\frac{\ln \vz^{-2}}{ a_{2^l}\ln \oz^{-1}}\Big\}.\notag
\end{align*}

It follows that \begin{align}\ln\,n(\vz,d)&\le
\ln\Big(\#\Big\{\h\in \Bbb Z^4\ |\ \sum\limits_{k=1}^4
|h_k|^{b_*}\le \frac{\ln \vz^{-2}}{a_1\ln \oz^{-1}}\Big\}\Big)\notag\\
&\qquad + \sum_{l=2}^{\lceil \log_2 d \rceil-1}\ln\Big(
\#\Big\{\h\in \Bbb Z^{2^l}\ \big|\
\sum_{k=2^l+1}^{2^{l+1}}|h_k|^{b_*}\le \frac{\ln
\vz^{-2}}{2^{l\frac{1-s}s}\hat h(2^l)\ln \oz^{-1}}\Big\}\Big)\notag\\
&\le \ln\Big(\#\Big\{\h\in \Bbb Z^4\ |\ \sum\limits_{k=1}^4
|h_k|^{b_*}\le \frac{\ln \vz^{-2}}{a_1\ln \oz^{-1}}\Big\}\Big)\notag\\
&\qquad +  \sum_{l=2}^{\lceil \log_2 d \rceil-1}\ln\Big(
\#\Big\{\h\in \Bbb Z^{2^l}\ \big|\
\sum_{k=2^l+1}^{2^{l+1}}|h_k|^{b_*}\le \frac{\ln
\vz^{-2}}{2^{l\frac{1-s}s} h(2^l)\ln \oz^{-1}}\Big\}\Big)\notag\\
&=: I_{1,\vz}+ \sum_{l=2}^{\lceil \log_2 d
\rceil-1}I_{l,\vz}.\label{3.8}
\end{align}

 By \eqref{3.0} we have
\begin{equation}\label{3.9}\frac{I_{1,\vz}}{(\ln\vz^{-1})^s+d^t}\le \frac {4C\ln
 \big(\frac{2\ln
\vz^{-2}}{a_1\ln\oz^{-1}}\big)}{(\ln\vz^{-1})^s+d^t}\longrightarrow
0\end{equation} as $\vz^{-1}+d\to\infty$.

We set $$m_{l,\vz}=\frac{\ln \vz^{-2}}{2^{l\frac{1-s}s} h(2^l)\ln
\oz^{-1}}.$$

It is easy to see that the sequence
\begin{equation*}\{d_{l,\vz}\}\equiv \{\frac{m_{l,\vz}}{2^l}\}=\{\frac{\ln \vz^{-2}}{2^{l/s}
h(2^l)\ln\oz^{-1}}\}\end{equation*} satisfies
\begin{equation}\label{3.10}2^{-(1+1/s)}\le \frac{d_{l+1,\vz}}{d_{l,\vz}}= \frac{h(2^l)}{2^{1/s}h(2^{l+1})}
\le 2^{-1/s}<1,\end{equation}
$$\lim_{l\to\infty}d_{l,\vz}=
\lim_{l\to\infty}\frac{m_{l,\vz}}{2^l}=0, \ \ {\rm and} \ \
m_{2,\vz}\ge 4\  {\rm for\ sufficiently \ large}\ \vz^{-1}.$$ Then
there exists an $l_0\ge 2$ such that
\begin{equation}\label{3.11}d_{l_0,\vz}=\frac{m_{l_0,\vz}}{2^{l_0}}\ge 1\ \ \ {\rm and}\ \ \
d_{l_0+1,\vz}=\frac{m_{l_0+1,\vz}}{2^{l_0+1}}<1.\end{equation} It
follows that
\begin{equation}\label{3.12} d_{l,\vz}\le 2^{(1+1/s)(l_0+1-l)}d_{l_0+1,\vz}\le
2^{(1+1/s)(l_0+1-l)}\ \ {\rm for}\ \ l\le l_0,\end{equation}
\begin{equation}\label{3.13}\frac 1{d_{l,\vz}}\le 2^{(1+1/s)(l-l_0)} \frac1 {d_{l_0,\vz}}\le 2^{(1+1/s)(l-l_0)}\ \ {\rm for}\ \ l>
l_0,
\end{equation}
and $$1\le (d_{l_0,\vz})^s=\frac{(\ln \vz^{-2})^s}{2^{l_0}
(h(2^{l_0}))^s(\ln\oz^{-1})^s}\le 2^{1+s}.$$It follows that
\begin{equation}\label{3.14}2^{l_0}\le \frac{(2\ln \vz^{-1})^s}{
(h(2^{l_0}))^s(\ln\oz^{-1})^s},\end{equation}and $h(2^{l_0})$
tends to $\infty$ as $\vz^{-1}\to \infty$.

We note that  $I_{l,\vz}=0$ if $m_{l,\vz}< 1$.
  By \eqref{3.0} we have
\begin{equation}\label{3.15}I_{l,\vz}\le C\Bigg\{\begin{array}{ll}m_{l,\vz}\ln \big(\frac{2}{d_{l,\vz}}\big), & d_{l,\vz}\le 1,\\
2^l\ln \big(2d_{l,\vz}\big), &
d_{l,\vz}\ge1.\end{array}\end{equation} Hence, by \eqref{3.15},
\eqref{3.12}, \eqref{3.13}, and \eqref{3.11} we have
\begin{align*} \sum_{l=2}^{\lceil \log_2 d \rceil-1}I_{l,\vz}&\le
\sum_{l=2}^{l_0}C 2^l\ln (2d_{l,\vz})+\sum_{l=l_0+1}^\infty
Cm_{l,\vz}\ln(\frac2{d_{l,\vz}})\\ &\le C\sum_{l=2}^{l_0}2^l
[(1+1/s)(l_0+1-l)+1]\ln 2 \\ +C\sum_{l=l_0+1}^\infty &\frac{\ln
\vz^{-2}}{2^{l\frac{1-s}s} h(2^{l_0})\ln \oz^{-1}}
[(1+1/s)(l-l_0)+1]\ln2\\&\le C_12^{l_0}+C_1 \frac{\ln
\vz^{-2}}{2^{l_0\frac{1-s}s} h(2^{l_0})\ln \oz^{-1}} \le C_2
2^{l_0}.
\end{align*}Hence, by \eqref{3.14} we have
$$ \frac{\sum\limits_{l=2}^{\lceil \log_2 d
\rceil-1}I_{l,\vz}}{(\ln\vz^{-1})^s+d^t}\le \frac{ C_3
2^{l_0}}{(\ln\vz^{-1})^s}\le \frac{C_32^s}{
(h(2^{l_0}))^s(\ln\oz^{-1})^s}\to 0 $$as $\vz^{-1}\to \infty$.
This, combining with \eqref{3.8} and \eqref{3.9} means that
EC-$(s,t)$-WT with $t\le1$ holds for APP if \eqref{1.8} holds.
Theorem 1.1  is proved. $\hfill\Box$

\newpage

\noindent{\it \textbf{Proof of Theorem \ref{thm2}}.}

\vskip 2mm

According to \cite[Theorem 3.1]{LXD}, we know that we have the
same results in the worst and average case settings under ABS
concerning EC-$(s,t)$-WT for $0<s\le 1$ and $t>0$.

\vskip 2mm

(i)  It follows that EC-$(s,t)$-WT always holds for $0<s\le 1$ and
$t>1$ for ABS. This yields that EC-$(s,t)$-WT holds  for $s> 1$
and $t>1$ for ABS, and by \eqref{2.2} also for NOR. Hence (i)
holds.

\vskip 2mm

(ii) If EC-$(s,1)$-WT with $s\ge 1$ holds for ABS or NOR, then
$(s,1)$-WT with $s\ge 1$ holds also for ABS or NOR. It follows
from \cite[Theorem 5.1]{LX3}  that $\lim\limits_{j\to\infty}
a_j=\infty.$

On the other hand, if $\lim\limits_{j\to\infty}a_j=\infty$, then
EC-WT holds for ABS or NOR and hence, EC-$(s,1)$-WT with $s\ge 1$
also holds for ABS or NOR. This completes the proof of (ii).

\vskip 2mm

(iii)
 EC-$(1,t)$-WT with $t<1$  holds for ABS iff \eqref{1.10} holds.
 (iii) is proved.

\vskip 2mm

(iv)   If  \eqref{1.11} holds, then
 EC-$(s,t)$-WT with  $s<1, t\le1$ holds for ABS, and
 also for NOR by \eqref{2.2}.

On the other hand, assume that  EC-$(s,t)$-WT with $s<1, t\le1$
holds for ABS or NOR.
 By \eqref{2.2}, we know that  EC-$(s,t)$-WT with
$s<1, t=1$ holds also for  NOR.

 Also by \eqref{2.2}, we have
\begin{equation}\label{3.16}\frac{\ln  n^{\rm avg, ABS}(\va
,d)}{(\ln\vz^{-1})^s+d}=\frac{[\ln (e^{\rm
avg}(0,d)\vz^{-1})]^s+d}{(\ln\vz^{-1})^s+d} \cdot \frac{\ln n^{\rm
avg, NOR}((e^{\rm avg}(0,d))^{-1}\vz ,d)}{[\ln (e^{\rm
avg}(0,d)\vz^{-1})]^s+d}.\end{equation} By \eqref{2.4} we have
$e^{\rm avg}(0,d)\vz^{-1}+d\to \infty$ iff $\vz^{-1}+d\to\infty$,
and
\begin{equation}\label{3.17}\frac{[\ln (e^{\rm
avg}(0,d)\vz^{-1})]^s+d}{(\ln\vz^{-1})^s+d}\le
\frac{2^s[\ln(e(0,d))]^s}{d}+\frac{2^s(\ln
\vz^{-1})^s+d}{(\ln\vz^{-1})^s+d} \le
M_1^s\oz^{sa_1}+2^s.\end{equation} Since EC-$(s,t)$-WT with $s<1,
t=1$ holds  for  NOR, we get
$$\lim_{\vz^{-1}+d\to\infty}\frac{\ln n^{\rm
avg, NOR}((e^{\rm avg}(0,d))^{-1}\vz ,d)}{[\ln (e^{\rm
avg}(0,d)\vz^{-1})]^s+d}=0,$$which combining with \eqref{3.16} and
\eqref{3.17}, yields that
$$\lim_{\vz^{-1}+d\to\infty}\frac{\ln  n^{\rm avg, ABS}(\va
,d)}{(\ln\vz^{-1})^s+d}=0.$$ It follows that EC-$(s,t)$-WT with
$s<1, t=1$ holds  for ABS. Hence \eqref{1.11} holds. (iv) is
proved.

\vskip 2mm

(v) If EC-$(s,t)$-WT with $s> 1$ and $t<1$ holds, then $(s,t)$-WT
with $s> 1$ and $t<1$ holds. It follows from \cite[Theorem
4.7]{CW}, we have \eqref{1.12}.

On the other hand, suppose taht \eqref{1.12} holds.  We want to
show that $(s,t)$-WT with $s> 1$ and $t<1$ holds under ABS or NOR.
By \eqref{2.2} it suffices to prove that for $s> 1$ and $t<1$,
\begin{equation}\label{3.18}\lim_{\vz^{-1}+d\to\infty}\frac{\ln  n^{\rm avg, ABS}(\va
,d)}{(\ln\vz^{-1})^s+d^t}=0.\end{equation}

  It follows from \cite[Eqution (3.12)]{CWZ} that for any $s_d\in
(0,1/2]$
\begin{equation*}n^{avg,\rm ABS}(\va,d)\le\va^{\frac{-2(1-s_d)}{s_d}}\Big(\sum_{k=1}^\infty
\lz_{d,k}^{1-s_d}\Big)^{\frac{1}{s_d}},\end{equation*}where
\begin{equation}\label{3.19} u_d:=\max(\oz^{a_d}, \frac{1}{2d}), \qquad
\text{and}\qquad
s_d:=\frac{1}{2}\big(\ln^+\frac{1}{u_d}\big)^{-1}, \qquad d\in
\Bbb N.
\end{equation}
Furthermore, if \eqref{1.12} holds, then it follows from
\cite[Equations (3.13) and (3.14)]{CWZ} that
\begin{equation}\label{3.20} \ln n^{\rm avg, ABS}(\va,d)\le \frac{2}{s_d}\ln \va^{-1}+\frac {e^{1/2}M_{1/2}}{s_d}
\sum_{k=1}^du_k,\end{equation} and
\begin{equation*}\lim\limits_{d\to\infty}\frac1{d^ts_d}\,{\sum_{k=1}^du_k}=0,\end{equation*}
which means that
$$\lim_{\vz^{-1}+d\to\infty}\frac{e^{1/2}M_{1/2}
\sum_{k=1}^du_k}{s_d((\ln\vz^{-1})^s+d^t)}\le
e^{1/2}M_{1/2}\lim_{d\to\infty}\frac{ \sum_{k=1}^du_k}{s_dd^t}
=0.$$
 In order to prove \eqref{3.18}, by \eqref{3.20} it suffices to prove
 that for $s>1$,
\begin{equation}\label{3.21} \lim_{\vz^{-1}+d\to\infty}\frac{2\ln
\vz^{-1}}{s_d((\ln\vz^{-1})^s+d^t)}=0.\end{equation} By
\eqref{3.19} we have
\begin{equation}\label{3.22}\frac1{s_d}=2\ln^+ \frac1{u_d}\le
2\ln^+(2d).\end{equation}
 For $s>1$, by the
Young inequality $ab\le \frac{a^p}p+\frac{b^{p'}}{p'},\ a,b\ge0, \
1/p+1/p'=1$ with $p=\frac{1+s}2,\ p'=\frac{s+1}{s-1}$ we have
$$\lim_{\vz^{-1}+d\to\infty}\frac
{\ln^+(2d)\ln(\va^{-1})}{(1+\ln\vz^{-1})^s+d^t}=
\lim_{\vz^{-1}+d\to\infty}\frac
{\frac{(\ln\va^{-1})^{\frac{s+1}2}}p+ \frac
{(\ln^+(2d))^{p'}}{p'}} {(\ln\vz^{-1})^s+d^t}=0,$$ which combining
\eqref{3.22}, gives \eqref{3.21}. This finishes the proof of (v).

\vskip 2mm

The proof of  Theorem \ref{thm2} is completed. $\hfill\Box$

\end{document}